\documentclass[11pt,oneside,english,a4paper]{amsart}
\usepackage[T1]{fontenc}
\usepackage[utf8]{inputenc}
\usepackage[dvipsnames]{xcolor}
\usepackage[yyyymmdd,hhmmss]{datetime}
\usepackage[numbers,sort,compress]{natbib}
\usepackage{amstext,amsthm,amssymb}
\usepackage{xparse,mathrsfs,mathtools}
\usepackage[english]{babel}
\usepackage[
	unicode=true,
	pdfusetitle, pdfdisplaydoctitle=true,
	pdfpagemode=UseOutlines,
	bookmarks=true, bookmarksnumbered=false, bookmarksopen=true, bookmarksopenlevel=2,
	breaklinks=true,
	pdfencoding=unicode, psdextra,
	pdfcreator={},
	pdfborder={0 0 0}
]{hyperref}
\usepackage{lmodern,microtype}
\usepackage[a4paper]{geometry}
\usepackage[capitalise,nameinlink]{cleveref}
\usepackage[cal=dutchcal,bb=ams,scr=boondoxo]{mathalfa}

\hypersetup{
	colorlinks=false,
	pdfborder={0 0 0},
	pdfborderstyle={/S/U/W 0},
}

\newtheorem{thm}{Theorem}[section]
\newtheorem{lem}[thm]{Lemma}
\newtheorem{cor}[thm]{Corollary}
\newtheorem{prop}[thm]{Proposition}
\theoremstyle{remark}
\newtheorem{rem}[thm]{Remark}
\newtheorem{example}[thm]{Example}
\newtheorem*{example*}{Example}

\newcommand{\NN}{\mathbb{N}}
\newcommand{\ZZ}{\mathbb{Z}}
\newcommand{\QQ}{\mathbb{Q}}
\newcommand{\CC}{\mathbb{C}}
\newcommand{\integers}{\mathcal{O}}
\newcommand{\setM}{\mathscr{M}}
\let\ideal\mathfrak
\DeclareMathOperator{\rad}{rad}
\DeclareMathOperator{\lcm}{lcm}

\DeclarePairedDelimiter\braces{\lbrace}{\rbrace}
\DeclarePairedDelimiter\abs{\lvert}{\rvert}
\DeclarePairedDelimiter\floor{\lfloor}{\rfloor}
\NewDocumentCommand\set{ s o m o }{%
	\IfBooleanTF{#1}{\IfNoValueTF{#4}{\braces*{#3}}{\braces*{\,#3:#4\,}}}{%
	\IfNoValueTF{#2}{\IfNoValueTF{#4}{\braces{#3}}{\braces{\,#3:#4\,}}}{%
	\IfNoValueTF{#4}{\braces[#2]{#3}}{\braces[#2]{\,#3:#4\,}}}}%
}
\newcommand{\lr}[1]{#1}
\renewcommand{\lr}[1]{\left\lbrace\;{#1}\;\right.}

\crefname{section}{§}{§§}
\crefname{figure}{Figure}{Figures}
\crefformat{equation}{#2(#1)#3}
\crefformat{enumi}{#2(#1)#3}
\numberwithin{equation}{section}

\usepackage{xifthen,ifthen}
\makeatletter
\newcounter{@ToDo}
\newcommand{\todo@helper}[1]{%
	({\color{blue}TODO~\arabic{@ToDo}: {#1\@addpunct{.}}})%
}
\newcommand{\todo}[1]{\stepcounter{@ToDo}%
	\relax\ifmmode\text{\todo@helper{#1}}%
	\else\todo@helper{#1}\fi%
}
\makeatother

\usepackage{tikz}
\usetikzlibrary{calc}
\tikzset{
	my tree style/.style={
		baseline={([yshift=-.5ex]current bounding box.center)},
		level distance=7mm,
		level 1/.style={sibling distance=55mm},
		level 2/.style={sibling distance=55mm},
		level 3/.style={sibling distance=30mm},
		level 4/.style={sibling distance=25mm},
		level 5/.style={sibling distance=25mm},
		every node/.style={inner sep=2pt},
		->, >=stealth,
	}
}

\newcommand*{\doi}[1]{\href{http://dx.doi.org/#1}{doi: {\texttt{\footnotesize\color{teal}#1}}}}
\newcommand*{\cUrl}[1]{\href{#1}{{\texttt{\footnotesize\color{teal}#1}}}}

\title[Zero sums and Cilleruelo's conjecture]{Zero sums amongst roots and Cilleruelo's conjecture on the LCM of polynomial sequences}
\date{\today{}}

\subjclass[2020]{
	11N32.
}
\keywords{Least common multiple, polynomial sequence, Cilleruelo's conjecture, cyclotomic polynomial}

\author{Marc~Technau}
\address{%
	Marc~Technau\\%
	Institut für Analysis und Zahlentheorie\\%
	Graz University of Technology\\%
	Kopernikusgasse~24/II\\%
	8010~Graz\\%
	Austria}
\email{mtechnau@math.tugraz.at}

\begin{document}
\begin{abstract}
	We make progress on a conjecture of Cilleruelo on the growth of the least common multiple of consecutive values of an irreducible polynomial $f$ on the additional hypothesis that the polynomial be even.
	This strengthens earlier work of Rudnick--Maynard and Sah subject to that additional hypothesis when the degree of $f$ exceeds two.
	The improvement rests upon a different treatment of `large' prime divisors of $Q_f(N) = f(1)\cdots f(N)$ by means of certain zero sums amongst the roots of $f$.
	A similar argument was recently used by Baier and Dey with regard to another problem.
	The same method also allows for further improvements on a related conjecture of Sah on the size of the radical of $Q_f(N)$.
\end{abstract}
\maketitle

\section{Introduction}

The problem of understanding the arithmetic nature of the values taken by polynomials forms a very attractive part of number theory with a rich history and many long-standing unresolved questions.
For instance, even the innocuous seeming question whether the polynomial $X^2+1$ (or any other polynomial of degree $d\geq 2$ without any `local obstructions', for that matter) represents infinitely many primes remains open.
We refer the reader to~\cite{merikoski2023largest-prime-factor} for the current record on an approximation to this problem and the references therein for more history on the topic.

\medskip

The object of the present investigation is a problem for which, despite our limited understanding of large prime divisors in the values of polynomials, a satisfactory resolution has been given in the quadratic case.
Cilleruelo~\cite{cilleruelo2011lcm-quadratic} considered the problem of estimating the (logarithm of) the least common multiple of the numbers $f(1),\ldots,f(N)$.
Put
\begin{equation}\label{eq:def:Lf}
	L_f(N) \coloneqq \lcm \set{ f(1),\ldots,f(N) }.
\end{equation}
Cilleruelo's main result reads as follows:
\begin{thm}[Cilleruelo]\label{thm:Cilleruelo}
	For any irreducible quadratic polynomial $f\in\ZZ[X]$,
	\begin{equation}\label{eq:Cilleruelo}
		\log L_f(N) = N \log N + B_f N + o(N),
		\quad\text{as } N\to\infty,
	\end{equation}
	where $B_f$ is a certain explicitly computable constant depending on $f$.
\end{thm}
The proof of the above theorem revolves around a comparison of $L_f(N)$ with
\begin{equation}\label{eq:def:Qf}
	Q_f(N) \coloneqq \abs{ f(1)\cdots f(N) }.
\end{equation}
The particular effectiveness of this approach in connexion with \emph{quadratic} polynomials rests upon the elementary observation that sufficiently large primes $p\gg_f N$ may divide \emph{at most one} factor on the right hand side of~\cref{eq:def:Qf}, leading to the equation $\alpha_p(N) = \beta_p(N)$, where
\begin{gather}
	\label{eq:def:alpha:p}
	\alpha_p(N) \coloneqq \nu_p(Q_f(N)), \\
	\label{eq:def:beta:p}
	\beta_p (N) \coloneqq \nu_p(L_f(N)) = \max\set{ \nu_p(f(n)) }[ n\leq N ],
\end{gather}
and $\nu_p(k)$ denotes the largest $\nu\in\NN_0$ such that $p^\nu$ divides $k$ (see~\cite[Lemma~1]{cilleruelo2011lcm-quadratic}).
While the extraction of the secondary term $B_f N$ in~\cref{eq:Cilleruelo} relies on deep work on equidistribution of roots of quadratic congruences~\cite{toth200roots-quadr,duke1995equidistribution-roots-quadr}, and remains an impressive feat, it is first and foremost the above observation which renders obtaining asymptotics for $L_f(N)$ feasible.

\medskip

In pondering about possible extensions of \cref{thm:Cilleruelo} to higher-degree polynomials, Cilleruelo~\cite{cilleruelo2011lcm-quadratic} conjectured that
\begin{equation}\label{eq:Cilleruelo:Conj}
	\log L_f(N) \sim (d-1) {\mkern 1mu} N\log N
	\quad\text{as } N\to\infty,
\end{equation}
for all irreducible polynomials $f\in\ZZ$ of degree $d\geq 3$.
Maynard and Rudnick~\cite{maynard2021lower-bound-lcm} observed that a closer examination of Cilleruelo's arguments already furnishes the upper bound
\begin{equation}\label{eq:Cilleruelo:UpperBound}
	\log L_f(N) \lesssim (d-1) {\mkern 1mu} N\log N
	\quad(\text{as } N\to\infty),
\end{equation}
where $g(N) \lesssim h(N)$ means $\abs{g(N)} \leq (1+o(1)) h(N)$.
Turning to lower bounds, they established that
\[
	\log L_f(N) \gtrsim  \frac{1}{d} {\mkern 1mu} N\log N
	\quad(\text{as } N\to\infty).
\]
Shortly thereafter, Sah~\cite{sah2020improved-bound-lcm} gave a stunningly elegant refinement of the argument of~\cite{maynard2021lower-bound-lcm}, proving that
\[
	\log L_f(N) \gtrsim N\log N
	\quad(\text{as } N\to\infty).
\]
In order to achieve this, Sah employed a certain polynomial identity, in order to bound the number of times the power of a large prime can divide $f(n)$ for some $n\leq N$; see \cref{lem:MuBound:Sah} below.
The interested reader may compare this with the factorisation employed in the proof of Lemma~1 of~\cite{cilleruelo2011lcm-quadratic}, which may be viewed as the genesis of this approach.
Note also that Maynard and Rudnick employed a result quite similar to \cref{lem:MuBound:Sah} with $\nu=1$ and the resulting bound $d-1$ replaced by $d$, see~\cite[Lemma~2.2]{maynard2021lower-bound-lcm}.

\medskip

It should be clear by now that further progress towards~\cref{eq:Cilleruelo:Conj} necessitates a better understanding of the large prime divisors of $Q_f(N)$ as defined in~\cref{eq:def:Qf}.
While general results of the required strength seem very hard to obtain with the currently available technology, the goal of the present work is to show that for a quite general class of polynomials $f$, further progress can be made.
Our main result may be enunciated as follows:
\begin{thm}\label{thm:MainResult}
	Let $f\in\ZZ[X]$ be a monic irreducible polynomial of degree $d\geq 2$.
	Suppose that $f$ is even in the sense that $f(-X) = f(X)$.
	Then
	\begin{equation}\label{eq:MainResult}
		\log L_f(N) \geq \frac{2(d-1)}{d} N\log N - O(N).
	\end{equation}
\end{thm}

\begin{example*}
	\(\displaystyle
		\log L_{X^4-2}(N) \geq (3/2) N\log N - O(N).
	\)
\end{example*}

For the proof of \cref{thm:MainResult} we employ zero sums amongst the roots of $f$ in order to show that divisibility of $Q_f(N)$ by some large prime may only originate from particularly few factors in~\cref{eq:def:Qf}.
The key result to this effect is given as \cref{lem:AlgNT} below.
The author gratefully acknowledges the nice work of Baier and Dey~\cite{baier2020prime-powers} as the inspiration for this key lemma.

\medskip

In the course of his investigations, Sah~\cite{sah2020improved-bound-lcm} was lead to conjecture that the size of $L_f(N)$ should be governed by the primes that divide $Q_f(N)$ exactly once.
More precisely, he considered the quantity
\begin{equation}\label{eq:def:lf}
	\ell_f(N)
	\coloneqq \rad \lcm \set{ f(1),\ldots,f(N) }
	\quad(= \rad Q_f(N)).
\end{equation}
and conjectured that $\log \ell_f(N)$ satisfies the same asymptotics~\cref{eq:Cilleruelo:Conj} that is conjectured for $\log L_f(N)$.
(Note that, because of $\ell_f(N) \leq L_f(N)$ and~\cref{eq:Cilleruelo:UpperBound}, in trying to prove Sah's conjecture, only lower bounds are of interest.)
He proved that
\begin{equation}\label{eq:lf:Sah:LowerBound}
	\log \ell_f(N) \gtrsim \frac{2}{d} {\mkern 1mu} N\log N
	\quad(\text{as } N\to\infty).
\end{equation}
At this point, it seems worth pointing out that a function field analogue has also been studied; see~\cite{leumi2021lcm-function-field,entin2023lcm-function-field}.
In particular, Entin and Landsberg~\cite{entin2023lcm-function-field} showed that the corresponding analogues of Cilleruelo's conjecture and Sah's conjecture are equivalent.

Returning to the original setting involving polynomials over $\ZZ$, our approach also allows us to make some progress pertaining to Sah's conjecture.
\begin{thm}\label{thm:MainResult:Variant}
	On the hypotheses of \cref{thm:MainResult},
	\begin{equation}\label{eq:MainResult:Variant}
		\log \ell_f(N) \geq \frac{8(d-1)}{d(3d-2)} {\mkern 1mu} N\log N - O(N).
	\end{equation}
\end{thm}

\begin{example*}
	$\log \ell_{X^4-2}(N) \geq (3/5) N\log N - O(N)$.
\end{example*}

For the sake of comparison, we note that the difference of the coefficients in~\cref{eq:MainResult:Variant} and~\cref{eq:lf:Sah:LowerBound} is $2 (d-2)/(3d^2-2d)$, which is strictly positive for $d>2$.

\medskip

Furthermore, by taking full advantage of the machinery developed by Baier and Dey~\cite{baier2020prime-powers}, we are able to improve the bound in \cref{eq:lf:Sah:LowerBound} even further.
The price for this is the restriction to the very special family of cyclotomic polynomials of the shape $f = X^{2^\eta}+1$.
\begin{thm}\label{thm:MainResult:Cyclotomic}
	For fixed integers $\eta\geq 1$,
	\(\displaystyle
		\log \ell_{X^{2^\eta}+1}(N)
		\geq \frac{2^\eta-1}{\eta {\mkern 2mu} 2^{\eta-1}} N\log N - O_\eta(N)
	\).
\end{thm}
For the sake of comparing \cref{thm:MainResult:Cyclotomic} and \cref{thm:MainResult:Variant}, we note that the coefficient in front of $N\log N$ in \cref{thm:MainResult:Cyclotomic} is $\asymp 1/\log d$ in the degree $d$ of the polynomial while in \cref{thm:MainResult:Variant} the coefficient is $\asymp 1/d$, which decays much more quickly.
We also note that it is conceivable that the same method may yield a similar strengthening of \cref{thm:MainResult:Variant} for other families of polynomials.
We return to this in \cref{rem:FurtherGeneralisations} below.

\section{Notation}
Throughout the remainder, we use the following notation.
$f$ denotes a polynomial of degree $d\geq 2$.
The letter $p$ always denotes a prime number, and $\nu_p(k)$ denotes the exponent of $p$ occurring in the prime factorisation of the integer $k$.
$N$ is assumed to be an integer larger than $2$, which will generally be assumed to be large (depending on $f$).
The Landau notation and Vinogradov symbols have their usual meaning, and implied constants are generally allowed to depend on $f$.
Furthermore, we claim no uniformity on $f$ in any `little-o' notation or in `$\lesssim$' as introduced in~\cref{eq:Cilleruelo:UpperBound}.
The notations $Q_f(N)$, $L_f(N)$, $\ell_f(N)$, $\alpha_p(N)$ and $\beta_p(N)$ defined in the introduction are used throughout (see~\cref{eq:def:Qf}, \cref{eq:def:Lf}, \cref{eq:def:lf}, \cref{eq:def:alpha:p}, and \cref{eq:def:beta:p} respectively).

\section{Preparations}

In this section we collect some auxiliary results from~\cite{sah2020improved-bound-lcm}.
In particular, we shall require the following result on the growth of $Q_f(N)$:
\begin{lem}\label{lem:QfN:growth}
	$\log Q_f(N) = d {\mkern 1mu} N \log N + O(N)$.
\end{lem}
\begin{proof}
	This follows easily by noting that $f$ grows like a monomial;
	see, e.g., \cite[Eq.~(9)]{cilleruelo2011lcm-quadratic}, where this is spelled out for quadratic $f$.
	The general case follows \emph{mutatis mutandis}.
\end{proof}

For the purpose of comparing $Q_f(N)$ with $L_f(N)$ (or $\ell_f(N)$), the small prime divisors $p$ of $Q_f(N)$ cause problems, as they can be seen to have a very large multiplicity, arising from many of the values $f(1),\ldots,f(N)$ being divisible by $p$.
This results in the quantities $\alpha_p(N)$ and $\beta_p(N)$ (see~\cref{eq:def:alpha:p} and~\cref{eq:def:beta:p}) comparing quite unfavourably.
For that reason, it is advantageous to remove the contribution of small prime divisors.
This is accomplished via the following result:
\begin{lem}\label{lem:SmallPrimes:Contrib}
	For irreducible $f$ and fixed $c\geq1$,
	\(\displaystyle
		\log \prod_{p\leq c N} p^{\alpha_p(N)}
		= N \log N + O_c(N)
	\).
\end{lem}
\begin{proof}
	(This uses a close examination of the roots of $f$ modulo $p^\nu$ using a Hensel-type argument, combined with Chebotarev's density theorem, or earlier work of Kronecker or Frobenius.)
	The claim follows from \cite[Propositions~2.1 and~3.1]{sah2020improved-bound-lcm}.
	The first cited result is stated there in asymptotic notation without the $O(N)$ error term, but the required error term is spelled out in the proof.
\end{proof}

The proofs of \cref{thm:MainResult}, \cref{thm:MainResult:Variant}, and \cref{thm:MainResult:Cyclotomic} each depend on the same comparison principle.
For the sake of avoiding repetition, it seems appropriate to extract this into a separate result.
It reads as follows:

\begin{cor}\label{cor:LowerBoundFactory}
	Let $f$ be irreducible.
	Let $c,h\geq1$ be certain fixed constants, and write $\mathscr{L}_f(N)$ for either $L_f(N)$ or $\ell_f(N)$.
	Assume further that, for all large $N$,
	\begin{equation}\label{eq:LowerBoundFactory:Hypothesis}
		\mathscr{L}_f(N)^h \geq \prod_{p > c N} p^{\alpha_p(N)}.
	\end{equation}
	Then
	\[
		\log \mathscr{L}_f(N) \geq \frac{d-1}{h} {\mkern 1mu} N \log N - O_c(N).
	\]
\end{cor}
\begin{proof}
	Upon combining \cref{lem:QfN:growth} and \cref{lem:SmallPrimes:Contrib}, and recalling~\cref{eq:def:Qf}, we see that the right hand side of~\cref{eq:LowerBoundFactory:Hypothesis} is $(d-1) {\mkern 1mu} N \log N - O_c(N)$.
	The claimed result then follows upon taking logarithms and dividing by $h$.
\end{proof}

\section{The key lemma}

The next result is the key ingredient for our attack on~\cref{eq:Cilleruelo:Conj}.
\begin{lem}\label{lem:AlgNT}
	Let $f\in\ZZ[X]$ be a monic polynomial of degree at least two.
	Furthermore, suppose $u\geq 2$ is an integer with the following property:
		for any collection of $u$ distinct roots of $f$ in the complex numbers, there is a pair of roots summing to zero.
	Then, for $N$ sufficiently large in terms of $f$, and any prime $p > 2N$,
	\begin{equation}\label{eq:def:mu:p}
		\mu_p(N) \coloneqq \#\set{ n\leq N }[ p \text{ divides } f(n) ] \leq u-1.
	\end{equation}
\end{lem}
\begin{proof}
	Suppose to the contrary that there are arbitrarily large $N$, primes $p > 2N$ and sets $\setM\subseteq[1,N]\cap\NN$ with $u$ elements such that $p$ divides $f(m)$ for every $m\in\setM$.
	Consider the splitting field $L\subset\CC$ of $f$ and denote by $\integers$ the associated ring of integers.
	Fix some prime ideal $\ideal{P}\subset\integers$ lying above $p$.
	For any $m\in\setM$ we have
	\[
		p \mid f(m) = \prod_{\alpha} (m - \alpha),
	\]
	where the above factorisation takes place over $\integers$ with $\alpha$ ranging over the roots of $f$ in the complex numbers (taking multiplicities into account).
	Consequently
	\begin{equation}\label{eq:PrimeIdeal:Membership}
		\ideal{P} \ni m - \alpha_m
	\end{equation}
	for some complex root $\alpha_m$ of $f$.
	If the $\alpha_m$ thus obtained are all distinct as $m$ ranges over $\setM$, then our assumption guarantees that two of them sum to zero.
	If not, then there are two distinct elements of $\setM$ which are associated the same $\alpha$.
	Either way, there are two distinct $m,n\in\setM$ such that $\alpha_m \pm \alpha_n = 0$ for an appropriate sign choice.
	Therefore, by~\cref{eq:PrimeIdeal:Membership}, $\ideal{P}$ contains
	\[
		(m - \alpha_m) \pm (n - \alpha_n) = m \pm n \in \ZZ.
	\]
	Since $\ideal{P} \cap \ZZ = p\ZZ$, we derive that $p$ divides $m \pm n \neq 0$.
	However, this contradicts the assumption that $p > 2N \geq \abs{ m\pm n }$.
\end{proof}

\begin{rem}\label{rem:BaierDey:Proof}
	The reader may wish to compare \cref{lem:AlgNT} and its proof with the approach of Baier and Dey~\cite[§§~4--5]{baier2020prime-powers} (see \cref{thm:BaierDey} below).
	They use zero sums amongst roots of unity to descend divisibility down in the tower of subfields of the cyclotomic field $\QQ(\zeta)$, where $\zeta$ is a primitive $2^{n+1}$-th root of unity.
	While they use a more refined approach, we plainly descend divisibility down from $\integers$ to $\ZZ$.
\end{rem}

\begin{rem}\label{rem:PigeonHolingRoots}
	Let $f\in\ZZ[X]$ be a monic polynomial of degree $d\geq 2$.
	Suppose further that $f$ is even in the sense that $f(X) = f(-X)$.
	Then $f$ satisfies the hypotheses of \cref{lem:AlgNT} with $u = d/2+1$.
	Indeed, if $\alpha$ is a root of $f$, then so is $-\alpha$.
	Thus, the roots of $f$ pair up into mutually inverse pairs.
	(This also holds if $0$ is a root of $f$, as the evenness of $f$ guarantees that the multiplicity of the root $0$ is even.)
	In particular, any set of $u = d/2+1$ distinct roots of $f$ must contain two roots belonging to the same pair, and those two roots sum to zero, as desired.
\end{rem}

\begin{rem}
	In principle, instead of exploiting a zero sum of the shape $\alpha_m\pm\alpha_n$, one may also hope to be able to exploit zero sums of the shape
	\begin{equation}\label{eq:GeneralZeroSum}
		\sum_{m\in\mathscr{M}} b_m \alpha_m,
	\end{equation}
	with suitable integers $b_m$, in the proof of \cref{lem:AlgNT}.
	This may allow one to take $u$ smaller.
	(We thank Jakob Führer for the following example:)
	For instance, for
	\[
		f
		= X^6 - 2
		= \prod_{k\leq 6} (X - \alpha_k)
		\quad
		(\text{where } \alpha_k = \mathrm{e}^{2\pi\mathrm{i}k/6}\sqrt[3]{2}),
	\]
	\cref{rem:PigeonHolingRoots} only yields applicability of \cref{lem:AlgNT} for $u=4$, but, in fact, when allowing for more general zero sums of the shape~\cref{eq:GeneralZeroSum}, already $u=3$ is admissible.
	Indeed, any set $\setM$ of three distinct roots of $f$ either contains a pair of roots that sum to zero, or is, up to some sign choices, of the form $\setM = \set{ \alpha_k, \alpha_{k+1}, \alpha_{k+2} }$ for some integer $k$.
	However, in the latter case, $\alpha_k - \alpha_{k+1} + \alpha_{k+2} = 0$ constitutes a zero sum.
	The problem in trying to take advantage of this for improving \cref{lem:AlgNT} is that at the end of the proof one had to guarantee that $m\pm n\neq 0$.
	When using instead a zero sum of the shape~\cref{eq:GeneralZeroSum}, one faces the problem of guaranteeing that
	\[
		\sum_{m\in\mathscr{M}} b_m m \neq 0.
	\]
	In fact, this is an obstruction which already had to be navigated around in~\cite{baier2020prime-powers}.
\end{rem}

\section{Bounding \texorpdfstring{$L_f(Q)$}{L(Q)}: Proof of \texorpdfstring{\cref{thm:MainResult}}{Theorem\autoref{thm:MainResult}}}

\begin{proof}[Proof of \cref{thm:MainResult}]
	Observe that $f$ in \cref{thm:MainResult} is chosen such that it satisfies the hypothesis of \cref{lem:AlgNT} with $u=d/2+1$ (see \cref{rem:PigeonHolingRoots}).
	Consequently, if $p > 2N$, and $N$ is sufficiently large, we have
	\[
		\mu_p(N) \leq d/2.
	\]
	Hence, recalling~\cref{eq:def:beta:p} and~\cref{eq:def:mu:p},
	\[
		\beta_p(N)
		\geq \frac{1}{\mu_p(N)} \sum_{n\leq N} \nu_p(f(n))
		\geq \frac{1}{(d/2)} \alpha_p(N).
	\]
	Consequently,
	\begin{equation}\label{eq:MainResult:Proof}
		L_f(N)^{d/2} \geq \prod_{p>2N} p^{\alpha_p(N)}.
	\end{equation}
	In view of \cref{cor:LowerBoundFactory}, this proves \cref{eq:MainResult}.
\end{proof}

\section{Bounding \texorpdfstring{$\ell_f(Q)$}{ℓ(Q)}: Proof of \texorpdfstring{\cref{thm:MainResult:Variant}}{Theorem\autoref{thm:MainResult:Variant}}}

At this stage, we require the following key result from~\cite{sah2020improved-bound-lcm}:
\begin{lem}[Sah]\label{lem:MuBound:Sah}
	Let $f\in\ZZ[X]$ be an irreducible polynomial of degree $d$.
	Then, for any sufficiently large prime $p \gg_f N$ and any positive integer $\nu$,
	\begin{equation}\label{eq:MuBound:Sah}
		\mu_{p^\nu}(N) \coloneqq \#\set{ n\leq N }[ p^\nu \text{ divides } f(n) ] \leq d-\nu,
	\end{equation}
\end{lem}
\begin{proof}
	This is proved, \emph{inter alia}, in~\cite[§~4]{sah2020improved-bound-lcm}.
\end{proof}

By augmenting \cref{lem:MuBound:Sah} by \cref{lem:AlgNT} for a suitable range, we easily deduce the following sharpening of \cite[Lemma~4.1]{sah2020improved-bound-lcm}:
\begin{cor}\label{cor:MuBound:New}
	On the hypotheses of \cref{lem:AlgNT}, for $p\gg_f N$,
	\begin{equation}\label{eq:MuBound:New}
		\alpha_p(N) \leq (d-u/2) (u-1).
	\end{equation}
\end{cor}
\begin{proof}
	As $p\gg_f N$, we may suppose in the sequel $\mu_{p^\nu}(N) = 0$ for all $\nu \geq \deg f$.
	Then
	\begin{equation}\label{eq:Alpha:Telescoping}
		\alpha_p(N)
		= \sum_{\nu\geq 1} (\mu_{p^\nu}(N) - \mu_{p^{\nu+1}}(N)) {\mkern 1mu} \nu
		= \sum_{\nu<d} \mu_{p^\nu}(N).
	\end{equation}
	Moreover, by the trivial bound and \cref{lem:AlgNT},
	\begin{equation}\label{eq:MuBound:New:proof}
		\mu_{p^\nu}(N) \leq \mu_{p}(N) \leq u-1.
	\end{equation}
	We now split the last sum in~\cref{eq:Alpha:Telescoping} into two ranges, $\nu < d-u+1$ and $\nu \geq d-u+1$, and invoke our bound~\cref{eq:MuBound:New:proof} on the former range and Sah's bound~\cref{eq:MuBound:Sah} on the latter.
	An easy computation now yields~\cref{eq:MuBound:New}.
\end{proof}

We now follow Sah's proof of~\cref{eq:lf:Sah:LowerBound}, using \cref{cor:MuBound:New} in place of \cite[Lemma~4.1]{sah2020improved-bound-lcm}.

\begin{proof}[Proof of \cref{thm:MainResult:Variant}]
	Recall \cref{rem:PigeonHolingRoots} and take $u=d/2+1$.
	By \cref{cor:MuBound:New}, there is a constant $c$ depending only on $f$ such that
	\begin{equation}\label{eq:MainResult:Variant:Proof}
		\ell_f(N)^{(d-u/2) (u-1)} \geq \prod_{p>cN} p^{\alpha_p(N)}.
	\end{equation}
	By \cref{cor:LowerBoundFactory}, this establishes \cref{eq:MainResult:Variant}.
\end{proof}

\section{Improvements for certain cyclotomic polynomials}

Our goal in this section is to give a proof of \cref{thm:MainResult:Cyclotomic}.
The key technical tool we require is the following result due to Baier and Dey~\cite{baier2020prime-powers}:

\begin{thm}[Baier and Dey]\label{thm:BaierDey}
	Let $\eta\geq 1$.
	Suppose that $1\leq n_1<n_2<\ldots<n_s\leq N$ and $\nu_1 \geq \nu_2 \geq \ldots \geq \nu_s \geq 1$ are such that
	\begin{equation}\label{eq:SystemOfCongruences}
		\lr{\begin{aligned}
			n_1^{2^\eta} + 1 &\equiv 0 \mod p^{\nu_1}, \\
			n_2^{2^\eta} + 1 &\equiv 0 \mod p^{\nu_2}, \\[-2mm]
			                 &\vdotswithin{\equiv}     \\[-2mm]
			n_s^{2^\eta} + 1 &\equiv 0 \mod p^{\nu_s}. 
		\end{aligned}}
	\end{equation}
	Write $\floor{\xi}$ for the largest integer not exceeding $\xi$.
	Suppose further that
	\begin{equation}\label{eq:BaierDey:rIneq}
		r \geq \floor*{ 2^{\eta-1-\floor{(\log \nu_r)/\log 2}} } + 1
	\end{equation}
	for some $r\in\set{1,2,\ldots,s}$.
	Then
	\(
		p \ll_\eta N
	\).
\end{thm}
\begin{proof}
	This is proved in~\cite[§§~4--5]{baier2020prime-powers} (see also \cref{rem:BaierDey:Proof} and \cref{rem:FurtherGeneralisations}).
\end{proof}

In order to better illustrate the underlying principle, we first study \cref{thm:MainResult:Cyclotomic} in the special case when $\eta=2$.

\begin{example}\label{example:Improvements:Via:BaierDey}
	Take $\eta=2$ and consider the polynomial $f=X^{2^\eta}+1$ of degree $d=4$.
	Below we list all (decreasing) tuples $\boldsymbol{\nu} = (\nu_1,\ldots,\nu_s)$ such that, given a system of congruences~\cref{eq:SystemOfCongruences} with some large prime $p \gg_\eta N$, no contradiction is obtained from \cref{lem:MuBound:Sah} alone.
	Note that the number of all such tuples is finite, as each such tuple satisfies $s\leq d-1$ and contains only positive integers not exceeding $d-1$.
	The resulting tuples are depicted in the following tree:
	\begin{equation}\label{eq:Tuples:Sah}
		\lr
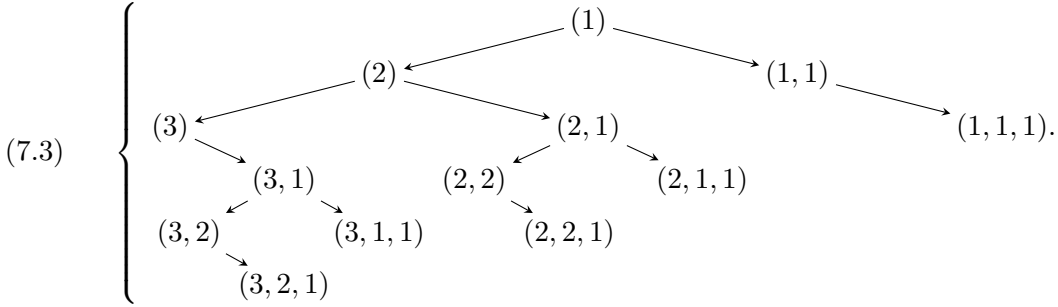
{\begin{tikzpicture}[my tree style]
			\node{\((1)\)}
				child{
					node{\((2)\)}
					child{
						node{\((3)\)}
						child[missing]{}
						child{
							node{\((3,1)\)}
							child{
								node{\((3,2)\)}
								child[missing]{}
								child{ node{\((3,2,1)\)} }
							}
							child{ node{\((3,1,1)\)} }
						}
					}
					child{
						node{\((2,1)\)}
						child{
							node{\((2,2)\)}
							child[missing]{}
							child{ node{\((2,2,1)\)} }
						}
						child{ node{\((2,1,1)\)} }
					}
				}
				child{
					node{\((1,1)\)}
					child[missing]{}
					child{ node{\((1,1,1)\).} } 
				}
			;
		\end{tikzpicture}}
	\end{equation}
	(For instance, $(3, 2, 2)$ is not listed, because $\nu_{p^2}(N) \geq 3 \nleq d-2$ in this case, contradicting~\cref{eq:MuBound:Sah}.)
	Similarly, we list the tuples for which no contradiction arises from~\cref{eq:BaierDey:rIneq}:
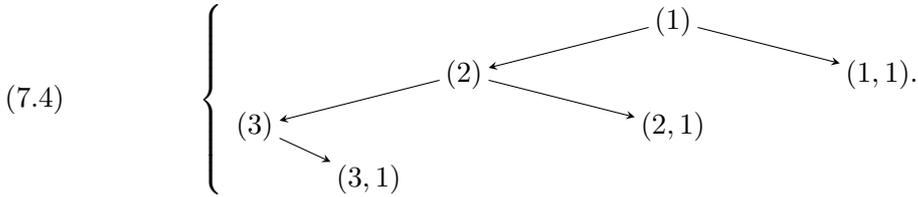
	\begin{equation}\label{eq:Tuples:BaierDey}
		\lr{\begin{tikzpicture}[my tree style]
			\node{\((1)\)}
				child{
					node{\((2)\)}
					child{
						node{\((3)\)}
						child[missing]{}
						child{ node{\((3,1)\)} }
					}
					child{ node{\((2,1)\)} }
				}
				child{ node{\((1,1)\).} } 
			;
		\end{tikzpicture}}
	\end{equation}
	Note that
	\[
		\max\set[\Big]{ \sum_r \nu_r }[ \boldsymbol{\nu} \text{ ranging over~\cref{eq:Tuples:Sah}} ]
		= 6
		= \frac{d(d-1)}{2},
	\]
	whereas
	\begin{equation}\label{eq:Improvements:Via:BaierDey}
		\max\set[\Big]{ \sum_r \nu_r }[ \boldsymbol{\nu} \text{ ranging over~\cref{eq:Tuples:BaierDey}} ]
		= \text{height of the tree~\cref{eq:Tuples:BaierDey}}
		= 4.
	\end{equation}
	From this we see that, in the proof of \cref{thm:MainResult:Variant}, we may replace the exponent on the left hand side of~\cref{eq:MainResult:Variant:Proof} (which evaluates to $5$ the situation currently under investigation) by $4$.
	\cref{cor:LowerBoundFactory} then yields
	\[
		\log \ell_{X^4+1}(N) \geq (3/4) N\log N - O(N).
	\]
	On the other hand, turning to the proof of~\cref{thm:MainResult}, we see that the exponent $d/2$ on the left hand side of~\cref{eq:MainResult:Proof} could be replaced by
	\[
		\max\set[\Big]{
			\frac{1}{ \displaystyle \max_r \nu_r } \sum_r \nu_r
		}[ \boldsymbol{\nu} \text{ ranging over~\cref{eq:Tuples:BaierDey}} ],
	\]
	although this equals $2 = d/2$ as well.
	Hence, we obtain no improvement for $\log L_{X^4+1}(N)$ over what \cref{thm:MainResult} already gives.
\end{example}

From the above example we see that the quantity given in~\cref{eq:Improvements:Via:BaierDey} and its generalisations for larger $\eta$ are crucial for obtaining \cref{thm:MainResult:Cyclotomic}.
The relevant computations have already been carried out in~\cite{baier2020prime-powers}, so that it suffices to cite the following result:

\begin{prop}[Baier and Dey]\label{prop:BaierDey:PartitionProblem}
	Let $\eta\geq 1$ and put $d = 2^\eta \:(= \deg(X^{2^\eta}+1))$.
	Then
	\[
		\max_{\boldsymbol{\nu}} \sum_r \nu_r
		= \frac{d \log d}{2 \log 2}
		= \eta {\mkern 2mu} 2^{\eta-1},
	\]
	where $\boldsymbol{\nu} = (\nu_1,\ldots,\nu_s)$ ranges over all (decreasing) tuples, such that, given a system of congruences~\cref{eq:SystemOfCongruences} with some large prime $p \gg_\eta N$, no contradiction is obtained from \cref{thm:BaierDey} alone.
\end{prop}
\begin{proof}
	This is~\cite[Lemma~8]{baier2020prime-powers}.
\end{proof}

\begin{proof}[Proof of \cref{thm:MainResult:Cyclotomic}]
	This is already outlined in \cref{example:Improvements:Via:BaierDey} for $\eta=2$.
	For the general case, one argues as in the example, and replaces~\cref{eq:Improvements:Via:BaierDey} by the result furnished by \cref{prop:BaierDey:PartitionProblem}.
	The resulting improvement of~\cref{eq:MainResult:Variant:Proof} is
	\[
		\ell_f(N)^{\eta {\mkern 2mu} 2^{\eta-1}} \geq \prod_{p>cN} p^{\alpha_p(N)}.
	\]
	Hence, \cref{cor:LowerBoundFactory} implies the desired result.
\end{proof}

\begin{rem}\label{rem:FurtherGeneralisations}
	A closer inspection of the arguments in~\cite[§§~4--5]{baier2020prime-powers} shows that one can hope to prove variants of \cref{thm:BaierDey} also for certain polynomials other than $f = X^{2^\eta}+1$, with the quality of the final result depending on the lattice of intermediate fields between the splitting field of $f$ and $\QQ$.
	In this way, one should be able to obtain further improvements of \cref{thm:MainResult:Variant} for certain families of polynomials not covered by \cref{thm:MainResult:Cyclotomic}.
	We have chosen not to pursue this further, as the additional clarity gained by being able to readily re-purpose results from~\cite{baier2020prime-powers} seemed to outweigh the additional complications entailed by striving for slightly more general results.
\end{rem}

\section*{Acknowledgements}

The author takes pleasure in thanking the joint
	FWF--ANR
project
	\emph{ArithRand} ({FWF~I~4945-N} and {ANR-20-CE91-0006})
for financial support.
The author thanks Benjamin Klahn and Jakob Führer for several stimulating conversations on the topic.
Further thanks are due to Jörn Steuding, who made possible a research stay of the author in Würzburg, where most of this work was carried out.



\vfill%
\end{document}